\documentclass[11pt]{amsart} 
\usepackage[foot]{amsaddr}

\usepackage{amsmath,amssymb,bm,float,color}

\usepackage{amsmath}
\usepackage{graphicx}

\usepackage{mathrsfs}
\usepackage{bbm}
\usepackage{bm}
\usepackage{amsfonts,amssymb}
\usepackage{multirow}
\usepackage{lineno}
\usepackage{color}

\numberwithin{equation}{section}
 \newtheorem{theorem}{Theorem}[section]
 \newtheorem{lemma}[theorem]{Lemma}

\def\3bar{{|\hspace{-.02in}|\hspace{-.02in}|}}

\def\cal#1{{\mathcal #1}}

\def\bb{{\mathbf{b}}}

\def\bv{{\mathbf{v}}}
\def\bn{{\mathbf{n}}}

\def\beta{{\boldsymbol{\eta}}}
\def\bvarphi{{\boldsymbol{\varphi}}}

\newtheorem{algorithm}{Weak Galerkin Algorithm}[section]

\numberwithin{equation}{section}

\def\3bar{{|\hspace{-.02in}|\hspace{-.02in}|}}

 \def\cal#1{\mathcal{#1}}
\def\ad#1{\begin{aligned}#1\end{aligned}}  \def\b#1{\mathbf{#1}} 
\def\a#1{\begin{align*}#1\end{align*}} \def\an#1{\begin{align}#1\end{align}} \def\t#1{\hbox{#1}}
\def\p#1{\begin{pmatrix}#1\end{pmatrix}}

\begin{document}

\title []{A Simple Weak Galerkin Finite Element Method for Convection-Diffusion-Reaction Equations on Nonconvex Polytopal Meshes}

  \author {Chunmei Wang$\dag$}\thanks{$\dag$ Corresponding author. } 
  \address{Department of Mathematics, University of Florida, Gainesville, FL 32611, USA. }
  \email{chunmei.wang@ufl.edu} 
 
\author {Shangyou Zhang}
\address{Department of Mathematical Sciences,  University of Delaware, Newark, DE 19716, USA}   \email{szhang@udel.edu}

\begin{abstract} 
This article introduces a simple weak Galerkin (WG) finite element method for solving convection-diffusion-reaction equation. The proposed method offers significant flexibility by supporting discontinuous approximating functions on general nonconvex polytopal meshes. We establish rigorous error estimates within a suitable norm. Finally, numerical experiments are presented to validate the theoretical convergence rates and demonstrate the computational efficiency of the approach.
\end{abstract}


  
\maketitle 
\section{Introduction}

The objective of this paper is to develop a simple weak Galerkin (WG) finite element method for convection-diffusion-reaction equations. This numerical framework is designed for applicability to general partitions consisting of both convex and non-convex polytopal meshes. For clarity of exposition, we focus on the steady-state convection-diffusion-reaction problem, which seeks an unknown scalar function $u$ such that:
\begin{equation}\label{model}
\begin{split}
-\rho\Delta u + \nabla \cdot (\mathbf{b} u) + cu =& f \quad \text{in } \Omega, \\
u =& 0 \quad \text{on } \partial\Omega,
\end{split}
\end{equation}
where $\rho > 0$ represents the diffusion coefficient and $\Omega \subset \mathbb{R}^d$ ($d=2, 3$) is a bounded polygonal or polyhedral domain with boundary $\partial\Omega$. To ensure the well-posedness of the problem, we assume   $\mathbf{b}, c, f$ possess sufficient regularity. Specifically, we assume $\mathbf{b} \in [W^{1, \infty}(\Omega)]^d$ and the coercivity condition $c + \frac{1}{2} \nabla \cdot \mathbf{b} \geq c_0 > 0$ holds for a constant $c_0$ (cf. \cite{32}).

The corresponding variational formulation of \eqref{model} is to find $u \in H_0^1(\Omega)$ such that:
\begin{equation}\label{weak}
(\rho\nabla u, \nabla v) + (\nabla \cdot (\mathbf{b} u), v) + (cu, v) = (f, v), \quad \forall v \in H_0^1(\Omega),
\end{equation}
where $H_0^1(\Omega)$ denotes the standard Sobolev space of functions with vanishing trace on $\partial\Omega$.

It is well established that the linear elliptic equation \eqref{model} presents significant numerical challenges in the singularly perturbed regime, characterized by $0 < \rho \ll 1$. In this context, the solution typically exhibits boundary or interior layers, narrow regions where the solution or its gradients undergo rapid variations. Conventional numerical schemes often suffer from spurious oscillations and fail to yield accurate approximations unless the computational mesh is refined to the scale of the singular perturbation parameter.

To mitigate these issues, various stabilization strategies have been proposed, which may be broadly categorized into fitted mesh methods and fitted operator methods. Historical efforts in mesh optimization include the work of Bakhvalov \cite{5}, while the piecewise-equidistant meshes introduced by Shishkin \cite{34} remain a cornerstone of singular perturbation analysis. Furthermore, adaptive refinement techniques have been extensively utilized over the past several decades \cite{4, 11} to resolve localized sharp features \cite{30}.

Recent decades have witnessed the rapid development of numerical methods utilizing discontinuous discrete spaces, most notably the discontinuous Galerkin (DG) and weak Galerkin (WG) methods. In the context of convection-dominated flows, DG methods leverage natural upwinding to provide inherent stabilization \cite{20, 32}; see also \cite{2, 6, 8, 15, 17, 19}. The WG method represents a powerful extension of this paradigm, approximating differential operators through a framework that parallels the theory of distributions for piecewise polynomials. By employing specifically designed stabilizers, the WG method relaxes the global regularity requirements of traditional finite element methods. The efficacy of WG has been demonstrated across a broad spectrum of partial differential equations (PDEs) \cite{wg1, wg2, wg3, wg4, wg5, wg6, wg7, wg8, wg9, wg10, wg11, wg12, wg13, wg14, wg15, wg16, wg17, wg18, wg19, wg20, wg21, itera, wy3655}. 

A significant advancement within this field is the Primal-Dual Weak Galerkin (PDWG) method, which addresses problems that are otherwise intractable for classical techniques \cite{pdwg1, pdwg2, pdwg3, pdwg4, pdwg5, pdwg6, pdwg7, pdwg8, pdwg9, pdwg10, pdwg11, pdwg12, pdwg13, pdwg14, pdwg15}. The PDWG approach formulates the numerical solution as a constrained minimization of a functional, where the constraints represent the weak form of the PDE. This results in an Euler-Lagrange system coupling a primal variable with a dual variable (Lagrange multiplier), yielding a symmetric saddle-point scheme. PDWG methods have been specifically extended to convection-diffusion equations in \cite{pdwg4, pdwg9, pdwg10}.

In the present work, we introduce a simple WG formulation that is applicable to general convex and non-convex polytopal meshes. The fundamental analytical component of this method is the utilization of bubble functions. While this necessitates higher-degree polynomials for the computation of the discrete weak gradient  and discrete weak divergence, the scheme preserves the global sparsity and dimensions of the resulting stiffness matrix.  

The remainder of this paper is structured as follows. Section 2 provides a concise review of the weak gradient and weak divergence operators. Section 3 introduces the simple WG scheme for the convection-diffusion-reaction equation. The existence and uniqueness of the discrete solution are established in Section 4. Section 5 is devoted to the derivation of the error equation. The error estimates in the energy norm are presented in Section  6. Finally, numerical experiments are presented in Section 7 to validate the theoretical findings and demonstrate the performance of the proposed method.

\textit{Notation.} Standard notation for Sobolev spaces is employed throughout. For an open bounded domain $D \subset \mathbb{R}^d$ with a Lipschitz boundary, $(\cdot,\cdot)_{s,D}$, $|\cdot|_{s,D}$, and $\|\cdot\|_{s,D}$ denote the inner product, semi-norm, and norm in $H^s(D)$, respectively. The subscript $D$ is omitted when $D = \Omega$. For $s=0$, we simplify the notation to $(\cdot,\cdot)_D$ and $\|\cdot\|_D$.
 
\section{Discrete Weak Gradient and Discrete Weak Divergence}

In this section, we provide a concise overview of the weak gradient and weak divergence operators, as well as their corresponding discrete representations, following the framework established in \cite{wy3655}.

Let $T$ be a polytopal element with boundary $\partial T$. A weak function on $T$ is defined as a doublet $v = \{v_0, v_b\}$, where $v_0 \in L^2(T)$ and $v_b \in L^2(\partial T)$. Here, the components $v_0$ and $v_b$ represent the values of the function in the interior and on the boundary of $T$, respectively. In general, $v_b$ is assumed to be independent of the trace of $v_0$ on $\partial T$. A special case arises when $v_b = v_0|_{\partial T}$, in which case the weak function $v$ is uniquely determined by $v_0$ and may be denoted simply as $v = v_0$.

Let $W(T)$ denote the space of weak functions on $T$, defined as:
\begin{equation}\label{2.1}
W(T) =   \{ v = \{v_0, v_b\} : v_0 \in L^2(T), v_b \in L^2(\partial T)  \}.
\end{equation}

The weak gradient, denoted by $\nabla_w$, is a linear operator mapping $W(T)$ to the dual space of $[H^1(T)]^d$. For any $v \in W(T)$, the weak gradient $\nabla_w v$ is defined as a bounded linear functional such that:
\begin{equation*}\label{2.3}
(\nabla_w v, \bvarphi)_T = -(v_0, \nabla \cdot \bvarphi)_T + \langle v_b, \bvarphi \cdot \mathbf{n} \rangle_{\partial T}, \quad \forall \bvarphi \in [H^1(T)]^d,
\end{equation*}
where $\mathbf{n}$ is the unit outward normal vector to $\partial T$.

Analogously, the weak divergence operator $\nabla_w \cdot (\mathbf{b} v)$ is defined as a linear operator from $W(T)$ to the dual space of $H^1(T)$. For any $v \in W(T)$, the weak divergence $\nabla_w \cdot (\mathbf{b} v)$ is the bounded linear functional satisfying:
\begin{equation*}
(\nabla_w \cdot (\mathbf{b} v), w)_T = -(\mathbf{b} v_0, \nabla w)_T + \langle (\mathbf{b} \cdot \mathbf{n}) v_b, w \rangle_{\partial T}, \quad \forall w \in H^1(T).
\end{equation*}

For any non-negative integer $r$, let $P_r(T)$ be the space of polynomials of degree at most $r$ on $T$. The discrete weak gradient, denoted by $\nabla_{w, r, T}$, is a linear operator from $W(T)$ to $[P_r(T)]^d$. For any $v \in W(T)$, $\nabla_{w, r, T} v$ is the unique polynomial vector in $[P_r(T)]^d$ satisfying:
\begin{equation}\label{2.4}
(\nabla_{w, r, T} v, \bvarphi)_T = -(v_0, \nabla \cdot \bvarphi)_T + \langle v_b, \bvarphi \cdot \mathbf{n} \rangle_{\partial T}, \quad \forall \bvarphi \in [P_r(T)]^d.
\end{equation}
For a sufficiently smooth function $v_0 \in H^1(T)$, applying the standard integration by parts formula to the first term on the right-hand side of \eqref{2.4} yields:
\begin{equation}\label{2.4new}
(\nabla_{w, r, T} v, \bvarphi)_T = (\nabla v_0, \bvarphi)_T + \langle v_b - v_0, \bvarphi \cdot \mathbf{n} \rangle_{\partial T}, \quad \forall \bvarphi \in [P_r(T)]^d.
\end{equation}

The discrete weak divergence on $T$, denoted by $\nabla_{w, r, T} \cdot (\mathbf{b} v)$, is a linear operator from $W(T)$ to $P_r(T)$. For any $v \in W(T)$, $\nabla_{w, r, T} \cdot (\mathbf{b} v)$ is the unique polynomial in $P_r(T)$ satisfying:
\begin{equation}\label{div}
(\nabla_{w, r, T} \cdot (\mathbf{b} v), w)_T = -(\mathbf{b} v_0, \nabla w)_T + \langle (\mathbf{b} \cdot \mathbf{n}) v_b, w \rangle_{\partial T}, \quad \forall w \in P_r(T).
\end{equation}
Furthermore, for $v_0 \in H^1(T)$, applying integration by parts to \eqref{div} gives:
\begin{equation}\label{divnew}
(\nabla_{w, r, T} \cdot (\mathbf{b} v), w)_T = (\nabla \cdot (\mathbf{b} v_0), w)_T + \langle (\mathbf{b} \cdot \mathbf{n})(v_b - v_0), w \rangle_{\partial T},  \forall w \in P_r(T).
\end{equation}
  \section{Weak Galerkin Algorithm}\label{Section:WGFEM}

Let $\mathcal{T}_h$ be a finite element partition of the domain $\Omega \subset \mathbb{R}^d$ into polytopes. We assume that $\mathcal{T}_h$ satisfies the shape-regularity conditions as detailed in \cite{wy3655}. Let $\mathcal{E}_h$ denote the set of all edges (for $d=2$) or faces (for $d=3$) in $\mathcal{T}_h$, and let $\mathcal{E}_h^0 = \mathcal{E}_h \setminus \partial\Omega$ denote the set of all interior edges or faces. For each element $T \in \mathcal{T}_h$, let $h_T$ represent its diameter, and define the mesh size $h = \max_{T \in \mathcal{T}_h} h_T$.

Let $k \geq 0$ and $q \geq 0$ be non-negative integers such that $k \geq q$. For each element $T \in \mathcal{T}_h$, we define the local weak finite element space as:
\begin{equation}
V(k, q, T) =   \{ v = \{v_0, v_b\} : v_0 \in P_k(T), v_b \in P_q(e), e \subset \partial T  \}.
\end{equation}
By assembling the local spaces $V(k, q, T)$ over the partition $\mathcal{T}_h$ and enforcing a unique value for $v_b$ on each interior interface $e \in \mathcal{E}_h^0$, we define the global weak finite element space:
\begin{equation}
V_h =   \{ v = \{v_0, v_b\} : \{v_0, v_b\}|_T \in V(k, q, T), \forall T \in \mathcal{T}_h  \}.
\end{equation}
The subspace of $V_h$ with vanishing boundary values on $\partial\Omega$ is denoted by:
\begin{equation}
V_h^0 =   \{ v \in V_h : v_b = 0 \text{ on } \partial\Omega  \}.
\end{equation}

For simplicity and when no confusion arises, for any $v \in V_h$, we denote the discrete weak gradient $\nabla_{w, r, T} v$ and the discrete weak divergence $\nabla_{w, r, T} \cdot (\mathbf{b} v)$ simply as $\nabla_w v$ and $\nabla_w \cdot (\mathbf{b} v)$, respectively. These operators are computed element-wise according to \eqref{2.4} and \eqref{div}:
\begin{align*}
(\nabla_w v)|_T =& \nabla_{w, r, T}(v|_T), \quad \forall T \in \mathcal{T}_h, \\
(\nabla_w \cdot (\mathbf{b} v))|_T =& \nabla_{w, r, T} \cdot (\mathbf{b} v|_T), \quad \forall T \in \mathcal{T}_h.
\end{align*}

In this work, the polynomial degree $r$ for the discrete operators is chosen as $r = k - 1 + 2N$ for non-convex polytopal elements and $r = k - 1 + N$ for convex polytopal elements, where $N$ denotes the number of edges or faces of the element $T$. Further justification for this choice of $r$ is provided in \cite{autowang}.

We introduce the global $L^2$ inner product notation $(\cdot, \cdot) = \sum_{T \in \mathcal{T}_h} (\cdot, \cdot)_T$ and define the following bilinear form $a(\cdot, \cdot)$ on $V_h \times V_h$:
\an{\label{form} \ad{
    a(u, v) &= \rho (\nabla_w u, \nabla_w v) + (\nabla_w \cdot (\mathbf{b} u), v_0)\\
            &+\sum_{T\in \mathcal T_h} \langle \b b \cdot \b n (u_0-u_b), v_0-v_b \rangle_{\partial T^+}
               + (c u_0, v_0), } }
  where $\b n$ is the unit outward normal vector on $\partial T$, and on $\partial T^+$ faces 
     $\b b \cdot \b n>0$.

The simple weak Galerkin numerical scheme for the convection-diffusion-reaction problem \eqref{model} is formulated as follows:

\begin{algorithm} \label{PDWG1}
Find $u_h = \{u_0, u_b\} \in V_h^0$ such that
\begin{equation}\label{WG}
a(u_h, v) = (f, v_0), \quad \forall v = \{v_0, v_b\} \in V_h^0.
\end{equation}
\end{algorithm}

\section{Existence and Uniqueness of the Numerical Solution}\label{Section:Existence}

We begin by recalling that for a shape-regular finite element partition $\mathcal{T}_h$ of the domain $\Omega$, the following trace inequality holds for any $T \in \mathcal{T}_h$ and $\phi \in H^1(T)$ \cite{wy3655}:
\begin{equation}\label{tracein}
\|\phi\|^2_{\partial T} \leq C   ( h_T^{-1}\|\phi\|_T^2 + h_T \|\nabla \phi\|_T^2  ).
\end{equation}
Furthermore, if $\phi$ is a polynomial on the element $T$, the following discrete trace inequality holds \cite{wy3655}:
\begin{equation}\label{trace}
\|\phi\|^2_{\partial T} \leq C h_T^{-1}\|\phi\|_T^2.
\end{equation}

For any weak function $v = \{v_0, v_b\} \in V_h$, we define the discrete energy norm as follows:
\begin{equation}\label{3norm}
\begin{split}
\3bar v \3bar^2 =& \sum_{T \in \mathcal{T}_h}   \rho(\nabla_w v, \nabla_w v)_T  +\frac{1}{2}\langle \bb \cdot\bn (v_b-v_0), v_0-v_b\rangle_{\partial T^-}\\&+\frac{1}{2}\langle \bb \cdot\bn (v_0-v_b), v_0-v_b\rangle_{\partial T^+} +  (  (c + \frac{1}{2}\nabla \cdot \mathbf{b} ) v_0, v_0  )_T.
\end{split}
\end{equation}
Additionally, we define a discrete $H^1$ semi-norm by:
\begin{equation}\label{disnorm}
\begin{split}
    \|v\|_{1, h}^2 =& \sum_{T \in \mathcal{T}_h}   \rho \|\nabla v_0\|_T^2  +\frac{1}{2}\langle \bb \cdot\bn (v_b-v_0), v_0-v_b\rangle_{\partial T^-}\\&+\frac{1}{2}\langle \bb \cdot\bn (v_0-v_b), v_0-v_b\rangle_{\partial T^+}  +  ( (c + \frac{1}{2}\nabla \cdot \mathbf{b}) v_0, v_0  )_T \\& + h_T^{-1} \|v_0 - v_b\|_{\partial T}^2  .
\end{split}
\end{equation}

To establish the stability of the scheme, we utilize several preliminary results.

\begin{lemma}\cite{autowang}\label{norm1}
For any $v = \{v_0, v_b\} \in V_h$, there exists a constant $C$ such that:
\begin{equation*}
\|\nabla v_0\|_T \leq C \|\nabla_w v\|_T.
\end{equation*}
\end{lemma}

\begin{lemma}\cite{autowang}\label{phi}
For $v = \{v_0, v_b\} \in V_h$, let $\bm{\varphi} = (v_b - v_0) \mathbf{n} \varphi_{e_i}$, where $\mathbf{n}$ is the unit outward normal to the edge/face $e_i$. Then the following inequality holds:
\begin{equation}
\|\bm{\varphi}\|_T^2 \leq C h_T \int_{e_i} |v_b - v_0|^2 ds.
\end{equation}
\end{lemma}

\begin{lemma}\label{normeqva}
There exist positive constants $C_1$ and $C_2$ such that for any $v = \{v_0, v_b\} \in V_h$:
\begin{equation}\label{normeq}
C_1 \|v\|_{1, h} \leq \3bar v \3bar \leq C_2 \|v\|_{1, h}.
\end{equation}
\end{lemma}

\begin{proof}
Consider a potentially non-convex polytopal element $T$. We define an edge/face-based bubble function as $\varphi_{e_i} = \prod_{k \neq i} l_k^2(x)$. It can be verified that  (1) $\varphi_{e_i}=0$ on the edge/face $e_k$ for $k \neq i$, (2) there exists a subdomain $\widehat{e_i}\subset e_i$ such that $\varphi_{e_i}\geq \rho_1$ for some constant $\rho_1>0$. 

By extending $v_b$ and the trace of $v_0$ from the edge $e_i$ to the element $T$ (denoted still as $v_b$ and $v_0$, see \cite{autowang}), and choosing $\bm{\varphi} = (v_b - v_0) \mathbf{n} \varphi_{e_i}$ in \eqref{2.4new}, we obtain:
\a{ (\nabla_w v, \bm{\varphi})_T 
  & = (\nabla v_0, \bm{\varphi})_T + \langle v_b - v_0, \bm{\varphi} \cdot \mathbf{n} \rangle_{\partial T}
  \\ & = (\nabla v_0, \bm{\varphi})_T + \int_{e_i} |v_b - v_0|^2 \varphi_{e_i} ds. }
Applying the Cauchy-Schwarz inequality, Lemma \ref{phi}, and the properties of the bubble function, we have:
\begin{equation*}
h_T^{-1} \int_{e_i} |v_b - v_0|^2 ds \leq C (\|\nabla_w v\|^2_T + \|\nabla v_0\|^2_T) \leq C \|\nabla_w v\|^2_T.
\end{equation*}
Combining this with Lemma \ref{norm1} and definitions \eqref{3norm}--\eqref{disnorm} yields the lower bound $C_1 \|v\|_{1, h} \leq \3bar v \3bar$.

Next, from \eqref{2.4new}, Cauchy-Schwarz inequality and  the trace inequality \eqref{trace}, we have
$$
 \Big|(\nabla_{w} v, \bvarphi)_T\Big| \leq \|\nabla v_0\|_T \|  \bvarphi\|_T+
Ch_T^{-\frac{1}{2}}\|v_b-v_0\|_{\partial T} \| \bvarphi\|_{T},
$$
which yields
$$
\| \nabla_{w} v\|_T^2\leq C( \|\nabla v_0\|^2_T  +
 h_T^{-1}\|v_b-v_0\|^2_{\partial T}),
$$
 and further gives $$ \3bar v\3bar  \leq C_2\|v\|_{1, h}.$$

 This completes the proof of the lemma.
\end{proof}

\begin{lemma}\label{ineq}
For any $v \in V_h$, the bilinear form satisfies the following coercivity-like identity:
\begin{equation*}
a(v, v) = \3bar v \3bar^2.
\end{equation*}
\end{lemma}
\begin{proof} 
    It follows from \eqref{div} and the ususal integration by parts that 
   \begin{equation*} 
       \begin{split} 
       &\sum_{T\in {\cal T}_h}(\nabla_w \cdot(\bb v), v_0)_T \\=& \sum_{T\in {\cal T}_h}-(\bb v_0, \nabla v_0)_T+\langle \bb \cdot\bn v_b, v_0\rangle_{\partial T}\\ 
      =& \sum_{T\in {\cal T}_h} (\nabla\cdot \bb v_0,   v_0)_T+(\bb v_0, \nabla v_0)_T
        +\langle \bb \cdot\bn (v_b-v_0), v_0\rangle_{\partial T}\\ 
      =& \sum_{T\in {\cal T}_h} (\nabla\cdot \bb v_0,   v_0)_T -(\nabla_w\cdot(\bb v), v_0)
        +\langle \bb\cdot \bn v_b, v_0\rangle_{\partial T}
        \\ & +\langle \bb \cdot\bn (v_b-v_0), v_0\rangle_{\partial T}\\ 
      =& \sum_{T\in {\cal T}_h} (\nabla\cdot\bb v_0,   v_0)_T -(\nabla_w\cdot(\bb v), v_0)+\langle \bb \cdot\bn (v_b-v_0), v_0-v_b\rangle_{\partial T},\\ 
       \end{split} 
   \end{equation*}    
   where we used $v\in   V_h^0$ and $\sum_{T\in {\cal T}_h} \langle \bb \cdot\bn  v_b, v_b\rangle_{\partial T}=0$. 
   
   This gives 
    $$ 
    \sum_{T\in {\cal T}_h}(\nabla_w \cdot(\bb v), v_0)_T= \frac{1}{2}\sum_{T\in {\cal T}_h} (\nabla\cdot\bb v_0,   v_0)_T  +\frac{1}{2}\langle \bb \cdot\bn (v_b-v_0), v_0-v_b\rangle_{\partial T}, 
    $$ 
which, yields   
\begin{equation*} 
       \begin{split} 
    a(v, v)=& \sum_{T\in {\cal T}_h}\rho(\nabla_w v, \nabla_w v)_T+(\nabla_w \cdot(\bb v), v_0)_T+\langle \bb\cdot\bn(v_0-v_b), v_0-v_b\rangle_{\partial T^+}+(cv_0, v_0)_T\\ 
    =& \sum_{T\in {\cal T}_h}\rho(\nabla_w v, \nabla_w v)_T+\frac{1}{2}  (\nabla\cdot\bb v_0,   v_0)_T \\
    & +\frac{1}{2}\langle \bb \cdot\bn (v_b-v_0), v_0-v_b\rangle_{\partial T}+\langle \bb\cdot\bn(v_0-v_b), v_0-v_b\rangle_{\partial T^+}+(cv_0, v_0)_T\\ 
    =& \sum_{T\in {\cal T}_h}\rho(\nabla_w v, \nabla_w v)_T+ ((c+\frac{1}{2}\nabla\cdot \bb) v_0,   v_0)_T
    \\ &  +\frac{1}{2}\langle \bb \cdot\bn (v_b-v_0), v_0-v_b\rangle_{\partial T^-}+\frac{1}{2}\langle \bb \cdot\bn (v_0-v_b), v_0-v_b\rangle_{\partial T^+} \\  
= & \3bar v\3bar^2. 
    \end{split} 
   \end{equation*} 
   This completes the proof of the lemma. 
   \end{proof}

\begin{theorem} \ 
The weak Galerkin numerical scheme \eqref{WG} possesses a unique solution $u_h \in V_h^0$.
\end{theorem}

\begin{proof}
Since the system is linear and finite-dimensional, uniqueness implies existence. Let $u_h^{(1)}, u_h^{(2)} \in V_h^0$ be two solutions, and define the error $\eta_h = u_h^{(1)} - u_h^{(2)} \in V_h^0$. It follows that $a(\eta_h, v) = 0$ for all $v \in V_h^0$. Setting $v = \eta_h$, Lemma \ref{ineq} implies $\3bar \eta_h \3bar = 0$. By the norm equivalence in Lemma \ref{normeqva}, we have $\|\eta_h\|_{1, h} = 0$, which necessitates $\nabla \eta_0 = 0$ on each element $T$ and $\eta_0 = \eta_b$ on $\partial T$. Using the fact that $\nabla \eta_0=0$ on each $T$ gives $\eta_0=C$ on each $T$. This, together with $\eta_0=\eta_b$ on each $\partial T$ and  $\eta_b=0$ on $\partial \Omega$,   gives $\eta_0\equiv 0$ and  further $\eta_b\equiv 0$  and $\eta_h\equiv 0$ in the domain $\Omega$. Therefore, we have $u_h^{(1)}\equiv u_h^{(2)}$. This completes the proof of this theorem.
\end{proof}

\section{Error Equations} 

In this section, we derive the error equations that govern the relationship between the exact solution and the weak Galerkin approximation. We begin by defining the necessary projection operators.

On each element $T \in \mathcal{T}_h$, let $Q_0$ denote the $L^2$ projection onto the polynomial space $P_k(T)$. Similarly, for each edge or face $e \subset \partial T$, let $Q_b$ denote the $L^2$ projection operator onto $P_q(e)$. For any function $w \in H^1(\Omega)$, the $L^2$ projection into the weak finite element space $V_h$, denoted by $Q_h w$, is defined such that:
\begin{equation}
(Q_h w)|_T := \{Q_0(w|_T), Q_b(w|_{\partial T})\}, \quad \forall T \in \mathcal{T}_h.
\end{equation}
Furthermore, let $Q_r$ be the $L^2$ projection operator onto the space of piecewise polynomials of degree $r$. As established previously, the degree $r$ is chosen as $r = k - 1 + 2N$ for non-convex elements and $r = k - 1 + N$ for convex elements, where $N$ is the number of faces of the polytope $T$.

\begin{lemma}\label{Lemma5.1}
The discrete weak operators satisfy the following commutative properties for any $u \in H^1(\Omega)$:
\begin{align}
\nabla_w u =& Q_r (\nabla u), \label{pro} \\
\nabla_w \cdot (\mathbf{b} u) =& Q_r (\nabla \cdot (\mathbf{b} u)). \label{pro2}
\end{align}
\end{lemma}

\begin{proof}
For any $u \in H^1(T)$, applying the definition of the discrete weak gradient \eqref{2.4new} gives
\begin{equation*}
\begin{split}
 (\nabla_{w} u
  ,\bvarphi)_T=&(\nabla u, \bvarphi)_T+
  \langle u|_{\partial T}-u|_T, \bvarphi \cdot \bn \rangle_{\partial T}\\=&(\nabla u, \bvarphi)_T\\
  =&(Q_r\nabla u, \bvarphi)_T, 
\end{split}
\end{equation*}
for all $\bm{\varphi} \in [P_r(T)]^d$. Similarly, for the discrete weak divergence, applying \eqref{divnew} yields:
\begin{equation*}
\begin{split}
(\nabla_{w}\cdot 
  (\bb u), w)_T=&  (\nabla \cdot (\bb u),  w)_T+
  \langle \bb 
  \cdot \bn (u|_{\partial T}-u|_T),  w\rangle_{\partial T}\\
=  &(\nabla \cdot (\bb u),  w)_T\\
  =&(Q_r\nabla \cdot (\bb u), w)_T, 
\end{split}
\end{equation*}
for all $w \in P_r(T)$. This confirms the assertions of the lemma.
\end{proof}

Let $u$ and $u_h \in V_h^0$ be the exact solution to the convection-diffusion-reaction problem \eqref{model} and its numerical approximation obtained from the WG Algorithm \ref{PDWG1}, respectively. We define the error function $e_h$ as:
\begin{equation}\label{error}  
e_h = u - u_h. 
\end{equation}  

\begin{lemma}\label{errorequa} 
The error function $e_h$ defined in \eqref{error} satisfies the following error equation:
\begin{equation}\label{erroreqn} 
\begin{split}
&\sum_{T\in \mathcal{T}_h}  (\rho \nabla_w e_h, \nabla_w v)_T + (\nabla_w \cdot (\mathbf{b} e_h), v_0)_T \\&+\langle \bb\cdot\bn(e_0-e_b), v_0-v_b\rangle_{\partial T^+}  + (c e_h, v_0)_T   = \ell(u, v), \quad \forall v \in V_h^0, 
\end{split}
\end{equation} 
where the anti-consistency error functional $\ell(u, v)$ is given by:
\begin{equation*}
\ell(u, v) = \sum_{T\in \mathcal{T}_h} \langle \rho(I - Q_r) \nabla u \cdot \mathbf{n}, v_0 - v_b \rangle_{\partial T}. 
\end{equation*} 
\end{lemma}

\begin{proof} 
By applying the properties \eqref{pro} and \eqref{pro2}, and setting $\bm{\varphi} = Q_r \nabla u$ in the definition of the discrete weak gradient \eqref{2.4new}, we obtain:
\begin{equation*}\label{54} 
\begin{split} 
&\sum_{T\in \mathcal{T}_h}    (\rho \nabla_w u, \nabla_w v)_T + (\nabla_w \cdot (\mathbf{b} u), v_0)_T \\&+\langle \bb\cdot\bn(u|_T-u|_{\partial T}), v_0-v_b\rangle_{\partial T^+}+ (cu, v_0)_T    \\
=& \sum_{T\in \mathcal{T}_h}   (\rho Q_r \nabla u, \nabla_w v)_T + (Q_r \nabla \cdot (\mathbf{b} u), v_0)_T \\& + (cu, v_0)_T    \\
=& \sum_{T\in \mathcal{T}_h}    (\rho Q_r \nabla u, \nabla v_0)_T + \langle \rho Q_r \nabla u \cdot \mathbf{n}, v_b - v_0 \rangle_{\partial T}
\\ & + (Q_r \nabla \cdot (\mathbf{b} u), v_0)_T  + (cu, v_0)_T    \\
=& \sum_{T\in \mathcal{T}_h}    (\rho \nabla u, \nabla v_0)_T + \langle \rho Q_r \nabla u \cdot \mathbf{n}, v_b - v_0 \rangle_{\partial T}
 \\ & + (\nabla \cdot (\mathbf{b} u), v_0)_T  + (cu, v_0)_T    \\
=& \sum_{T\in \mathcal{T}_h}    (f, v_0)_T + \langle \rho \nabla u \cdot \mathbf{n}, v_0 \rangle_{\partial T} + \langle \rho Q_r \nabla u \cdot \mathbf{n}, v_b - v_0 \rangle_{\partial T} \\
=& \sum_{T\in \mathcal{T}_h} (f, v_0)_T + \sum_{T\in \mathcal{T}_h} \langle \rho(I - Q_r) \nabla u \cdot \mathbf{n}, v_0 - v_b \rangle_{\partial T}, 
\end{split} 
\end{equation*} 
where we have employed the model equation \eqref{model}, standard integration by parts, and the fact that $\sum_{T\in \mathcal{T}_h} \langle \rho \nabla u \cdot \mathbf{n}, v_b \rangle_{\partial T} = \langle \rho \nabla u \cdot \mathbf{n}, v_b \rangle_{\partial \Omega} = 0$ since $v_b = 0$ on $\partial \Omega$. 

Subtracting the numerical scheme \eqref{WG} from the above identity  yields:
\a{  &\quad \ \sum_{T\in \mathcal{T}_h}    (\rho \nabla_w e_h, \nabla_w v)_T + (\nabla_w \cdot (\mathbf{b} e_h), v_0)_T \\&+\langle \bb\cdot\bn(e_0-e_b), v_0-v_b\rangle_{\partial T^+}+ (c e_h, v_0)_T 
\\ &
   = \sum_{T\in \mathcal{T}_h} \langle \rho(I - Q_r) \nabla u \cdot \mathbf{n}, v_0 - v_b \rangle_{\partial T}. 
} 
This completes the proof of the lemma. 
\end{proof}

\section{Error Estimates}

The following lemma recalls established approximation properties of the projection operators $Q_r$ and $Q_0$ under the assumed shape-regularity of the partition.

\begin{lemma}\cite{wy3655}
Let ${\cal T}_h$ be a finite element partition of the domain $\Omega$ satisfying the shape-regularity assumptions specified in \cite{wy3655}. For any $0\leq s \leq 1$, $0\leq n \leq k$, and $0\leq m \leq r$, the following estimates hold:
\begin{eqnarray}\label{error1}
 \sum_{T\in {\cal T}_h}h_T^{2s}\|\nabla u- Q_r \nabla u\|^2_{s,T}&\leq& C  h^{2m}\|u\|^2_{m+1},\\
\label{error2}
\sum_{T\in {\cal T}_h}h_T^{2s}\|u- Q _0u\|^2_{s,T}&\leq& C h^{2n+2}\|u\|^2_{n+1}.
\end{eqnarray}
 \end{lemma}

\begin{lemma} \ 
Assume the exact solution $u$ of the convection-diffusion-reaction equation \eqref{model} possesses sufficient regularity such that $u\in H^{k+1} (\Omega)$. There exists a constant $C$, independent of the mesh size $h$, such that the following approximation estimate holds:
\begin{equation}\label{erroresti1}
\3bar u-Q_hu \3bar \leq C(1+\rho^{\frac{1}{2}})h^k\|u\|_{k+1}.
\end{equation}
\end{lemma}

\begin{proof}
By employing the identity \eqref{2.4new}, the Cauchy-Schwarz inequality, the trace inequalities \eqref{tracein}--\eqref{trace}, and the approximation estimate \eqref{error2} for $n=k$ with $s=0, 1$, we obtain
\begin{equation*}
\begin{split}
&\quad\sum_{T\in {\cal T}_h} (\rho\nabla_w(u-Q_hu), \bv)_T \\
=&\sum_{T\in {\cal T}_h} (\rho\nabla(u-Q_0u),  \bv)_T+ \langle\rho( Q_0u-Q_bu), \bv\cdot\bn\rangle_{\partial T}\\
&\leq \Big(\sum_{T\in {\cal T}_h}\|\rho^{\frac{1}{2}}\nabla(u-Q_0u)\|^2_T\Big)^{\frac{1}{2}} \Big(\sum_{T\in {\cal T}_h}\|\rho^{\frac{1}{2}}\bv\|_T^2\Big)^{\frac{1}{2}}\\&\quad+ \Big(\sum_{T\in {\cal T}_h} \|\rho^{\frac{1}{2}}(Q_0u-Q_bu)\|_{\partial T} ^2\Big)^{\frac{1}{2}}\Big(\sum_{T\in {\cal T}_h} \|\rho^{\frac{1}{2}}\bv\|_{\partial T}^2\Big)^{\frac{1}{2}}\\
&\leq\Big(\ \sum_{T\in {\cal T}_h} \|\rho^{\frac{1}{2}}\nabla(u-Q_0u)\|_T^2\Big)^{\frac{1}{2}}\Big(\sum_{T\in {\cal T}_h} \|\rho^{\frac{1}{2}}\bv\|_T^2\Big)^{\frac{1}{2}}\\&\quad+\Big(\sum_{T\in {\cal T}_h}h_T^{-1} \|\rho^{\frac{1}{2}}(Q_0u-u)\|_{T} ^2+h_T \|\rho^{\frac{1}{2}}(Q_0u-u)\|_{1,T} ^2\Big)^{\frac{1}{2}}  \\
  & \quad \ \cdot \Big(\sum_{T\in {\cal T}_h}Ch_T^{-1}\|\rho^{\frac{1}{2}}\bv\|_T^2\Big)^{\frac{1}{2}}\\
&\leq C\rho^{\frac{1}{2}} h^k\|u\|_{k+1}\Big(\sum_{T\in {\cal T}_h} \|\rho^{\frac{1}{2}}\bv\|_T^2\Big)^{\frac{1}{2}},
\end{split}
\end{equation*}
for any $\bv\in [P_r(T)]^d$. Setting $\bv= \nabla_w(u-Q_hu)$ yields
\begin{equation}\label{eee1}
 \sum_{T\in {\cal T}_h} \rho\|\nabla_w(u-Q_hu)\|_T^2\leq C\rho h^{2k}\|u\|_{k+1}^2.   
\end{equation}

Similarly, utilizing \eqref{divnew}, the Cauchy-Schwarz inequality, the trace inequalities \eqref{tracein}--\eqref{trace}, and the estimate \eqref{error2} for $n=k$ and $s=0, 1$, we have
\begin{equation*}
\begin{split}
&\quad\sum_{T\in {\cal T}_h} (\nabla_w\cdot (\bb (u-Q_hu)), w)_T \\
=&\sum_{T\in {\cal T}_h} (\nabla\cdot (\bb (u-Q_0u)),  w)_T+ \langle \bb \cdot \bn (Q_0u-Q_bu), w\rangle_{\partial T}\\
&\leq \Big(\sum_{T\in {\cal T}_h}\|\nabla\cdot (\bb (u-Q_0u))\|^2_T\Big)^{\frac{1}{2}} \Big(\sum_{T\in {\cal T}_h}\|w\|_T^2\Big)^{\frac{1}{2}}\\&\quad+ \Big(\sum_{T\in {\cal T}_h} \| \bb \cdot \bn (Q_0u-Q_bu)\|_{\partial T} ^2\Big)^{\frac{1}{2}}\Big(\sum_{T\in {\cal T}_h} \|w\|_{\partial T}^2\Big)^{\frac{1}{2}}\\
&\leq \Big(\sum_{T\in {\cal T}_h}\|\nabla\cdot (\bb (u-Q_0u))\|^2_T\Big)^{\frac{1}{2}} \Big(\sum_{T\in {\cal T}_h}\|w\|_T^2\Big)^{\frac{1}{2}}\\&\quad+ \Big(\sum_{T\in {\cal T}_h} h_T^{-1}\| \bb \cdot \bn (Q_0u- u)\|_{ T} ^2+h_T \| \bb \cdot \bn (Q_0u- u)\|_{1, T} ^2\Big)^{\frac{1}{2}}
\\ & \quad \cdot \Big(\sum_{T\in {\cal T}_h} h_T^{-1}\|w\|_{T}^2\Big)^{\frac{1}{2}}\\
&\leq Ch^k\|u\|_{k+1}\Big(\sum_{T\in {\cal T}_h} \|w\|_T^2\Big)^{\frac{1}{2}},
\end{split}
\end{equation*}
for any $w\in P_r(T)$. Choosing $w=\nabla_w\cdot (\bb (u-Q_hu))$ gives
\begin{equation}\label{eee2}
    \sum_{T\in {\cal T}_h}  \|\nabla_w\cdot \bb (u-Q_hu)\|^2_T\leq 
Ch^{2k}\|u\|_{k+1}^2.
\end{equation} 

Applying  the estimate \eqref{error2} with $n=k$ and $s=0$, we have
\begin{equation}\label{eee3}
\sum_{T\in {\cal T}_h} c\|u-Q_0u\|_T^2\leq Ch^{2(k+1)}\|u\|^2_{k+1}.
\end{equation}

Using the trace inequality \eqref{trace}, the estimate \eqref{error2} with $n=k$ and $s=0$ and $s=1$,
we have
\begin{equation}
    \begin{split}
        &\sum_{T\in {\cal T}_h} \langle \bb\cdot \bn (Q_bu-Q_0u),Q_bu-Q_0u \rangle_{\partial T} \\ \leq & C\|Q_bu-Q_0u\|^2_{\partial T}\\
        \leq& Ch_T^{-1}\|u-Q_0u\|^2_{T}+Ch_T \|u-Q_0u\|^2_{1, T}\\
        \leq &Ch^{2k+1} \|u\|^2_{k+1}.
    \end{split}
\end{equation}

Finally, invoking Lemma \ref{ineq} and combining the estimates \eqref{eee1}, \eqref{eee2}, and \eqref{eee3}, we arrive at
 $$\3bar u-Q_hu \3bar^2 =a(u-Q_hu, u-Q_hu)
 \leq C(\rho+1) h^{2k} \|u\|^2_{k+1}.$$
 This concludes the proof of the lemma.
\end{proof}

\begin{theorem}
Let $u \in H^{k+1}(\Omega)$ be the exact solution of the convection-diffusion-reaction equation \eqref{model}. There exists a constant $C$ such that the following error estimate holds in the discrete energy norm:
\begin{equation}\label{trinorm}
\3bar u-u_h\3bar \leq C(1+\rho^{\frac{1}{2}} )h^k\|u\|_{k+1}.
\end{equation}
\end{theorem}

\begin{proof}
To estimate the right-hand side of the error equation \eqref{erroreqn}, we apply the Cauchy-Schwarz inequality, the trace inequality \eqref{tracein}, the approximation result \eqref{error1} for $m=k$ with $s=0,1$, and the norm equivalence \eqref{normeq} to obtain
\begin{equation}\label{erroreqn1}
\begin{split}
&\Big|\sum_{T\in {\cal T}_h}\langle \rho(I- Q_r)\nabla u\cdot\bn, v_0-v_b\rangle_{\partial T} \Big|\\
\leq & C (\sum_{T\in {\cal T}_h}\|\rho(I- Q_r)\nabla u\cdot\bn\|^2_T+h_T^2\|\rho (I- Q_r)\nabla u\cdot\bn\|^2_{1, T})^{\frac{1}{2}} 
\\&\cdot(\sum_{T\in {\cal T}_h}h_T^{-1}\|v_0-v_b\|^2_{\partial T})^{\frac{1}{2}}  \\
\leq & C\rho h^k\|u\|_{k+1} \| v\|_{1,h} \\ 
\leq & C\rho h^k\|u\|_{k+1} \3bar v\3bar.
\end{split}
\end{equation}
Substituting \eqref{erroreqn1} into \eqref{erroreqn} yields
\an{\label{err} \ad{ &\quad \ 
\sum_{T\in {\cal T}_h}(\rho\nabla_w e_h, \nabla_w v)_T+(\nabla_w \cdot (\bb e_h), v_0)_T\\ &+\langle \bb\cdot\bn(e_0-e_b), v_0-v_b\rangle_{\partial T^+}    +(ce_0, v_0)_T
  \leq   C\rho h^k\|u\|_{k+1} \3bar  v\3bar. } } 

Using \eqref{ineq}, the Cauchy-Schwarz inequality, the triangle inequality, and setting $v=Q_hu-u_h$ in \eqref{err}, while considering the estimate  \eqref{erroresti1}, we find
\begin{equation*}
\begin{split}
& \3bar u-u_h\3bar^2 
=  a(u-u_h, u-u_h)\\
=  &  a(u-u_h, u-Q_hu)+a(u-u_h, Q_hu-u_h)\\
\leq &a( u-u_h, u-u_h) a(u-Q_hu, u-Q_hu) \\
  & + C\rho h^k\|u\|_{k+1} \3bar Q_hu-u_h\3bar  \\
\leq &\3bar u-u_h  \3bar  \3bar u-Q_hu\3bar\\
  &  + C\rho h^k\|u\|_{k+1}(\3bar Q_hu-u \3bar +\3bar  u-u_h\3bar ) \\
\leq &\3bar u-u_h  \3bar  C(1+\rho^{\frac{1}{2}
}) h^k\|u\|_{k+1} + C\rho h^k\|u\|_{k+1}  (1+\rho^{\frac{1}{2}
}) h^k\|u\|_{k+1}\\
&+C\rho h^k\|u\|_{k+1} \3bar u-u_h\3bar.
\end{split}
\end{equation*}
Rearranging the terms leads to the desired result:
\begin{equation*}
\begin{split}
 \3bar u-u_h\3bar  \leq C (1+\rho^{\frac{1}{2}
}) h^k\|u\|_{k+1}. 
\end{split}
\end{equation*} 

This completes the proof of the theorem.
\end{proof}

\section{Numerical experiments}

We solve the convection-diffusion-reaction equation \eqref{model} 
   on a square domain $\Omega= (-1,1)\times (-1,1)$, where  
\an{\label{b} \rho=1 \ \t{or} \ 10^{-6} \ \t{or} \ 10^{-9}, \quad
     \b b =\p{1\\1}, \quad c =1.  }
By choosing $f$ in \eqref{model}, the exact solution is, independent of $\rho$
    (i.e., no boundary/interior layers), 
\an{\label{u-1} u=\sin (\pi x) \sin (\pi y). }  

\begin{figure}[H]
 \begin{center}\setlength\unitlength{1.0pt}
\begin{picture}(360,120)(0,0)
  \put(15,108){$G_1$:} \put(125,108){$G_2$:} \put(235,108){$G_3$:} 
  \put(0,-420){\includegraphics[width=380pt]{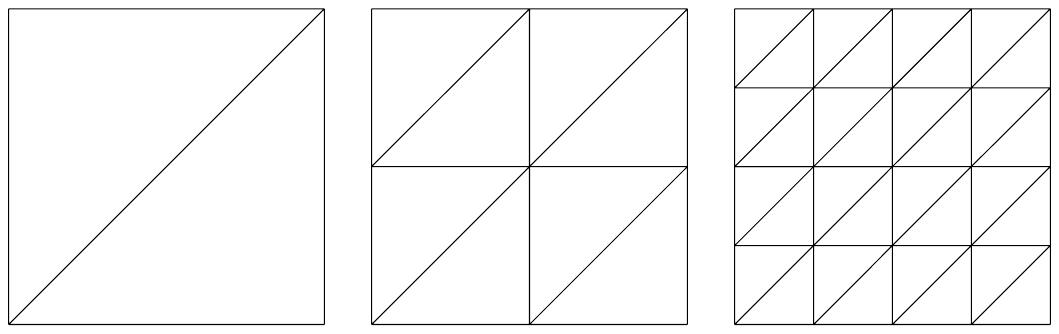}}  
 \end{picture}\end{center}
\caption{The triangular grids for the computation in Tables \ref{t1}--\ref{t4}. }\label{f21}
\end{figure}

The solution in \eqref{u-1} is approximated by the weak Galerkin finite element
   $P_k$-$P_k$/$P_{k+1}$ (for $\{u_0, u_b\}$/$\nabla_w$), $k= 1,2,3,4$, on triangular  
    grids shown in Figure \ref{f21}.
The errors and the computed orders of convergence are listed in Tables \ref{t1}--\ref{t4}.
The optimal order of convergence is achieved in every case.
   
  \begin{table}[H]
  \caption{By the $P_1$-$P_1$/$P_2$ element for \eqref{u-1}  on Figure \ref{f21} grids.} \label{t1}
\begin{center}  
   \begin{tabular}{c|rr|rr}  
 \hline 
$G_i$ &  $ \|Q_h  u -   u_h \| $ & $O(h^r)$ &  $ \sqrt{\rho}\| \nabla_w( Q_h u- u_h )\|  $ & $O(h^r)$\\
\hline 
   &\multicolumn{4}{c}{$\rho=1$ in \eqref{b} }\\  \hline 
 5&    0.164E-02 &  2.0&    0.108E+00 &  1.0\\
 6&    0.412E-03 &  2.0&    0.537E-01 &  1.0\\
 7&    0.103E-03 &  2.0&    0.269E-01 &  1.0\\
 \hline  &\multicolumn{4}{c}{$\lambda=10^{-6}$ in \eqref{b} }\\
 \hline 
 5&    0.142E-02 &  2.0&    0.204E-03 &  1.0\\
 6&    0.358E-03 &  2.0&    0.103E-03 &  1.0\\
 7&    0.896E-04 &  2.0&    0.515E-04 &  1.0\\ 
  \hline  
\end{tabular} \end{center}  \end{table}

  \begin{table}[H]
  \caption{By the $P_2$-$P_2$/$P_3$ element for \eqref{u-1}  on Figure \ref{f21} grids.} \label{t2}
\begin{center}  
   \begin{tabular}{c|rr|rr}  
 \hline 
$G_i$ &  $ \|Q_h  u -   u_h \| $ & $O(h^r)$ &  $ \sqrt{\rho}\| \nabla_w( Q_h u- u_h )\| $ & $O(h^r)$\\
\hline 
   &\multicolumn{4}{c}{$\rho=1$ in \eqref{b} }\\  \hline 
 5&    0.192E-04 &  3.0&    0.406E-02 &  2.0\\
 6&    0.239E-05 &  3.0&    0.101E-02 &  2.0\\
 7&    0.297E-06 &  3.0&    0.253E-03 &  2.0\\
 \hline  &\multicolumn{4}{c}{$\lambda=10^{-6}$ in \eqref{b} }\\
 \hline 
 5&    0.192E-04 &  3.0&    0.406E-02 &  2.0\\
 6&    0.239E-05 &  3.0&    0.101E-02 &  2.0\\
 7&    0.297E-06 &  3.0&    0.253E-03 &  2.0\\
  \hline  
\end{tabular} \end{center}  \end{table}

  \begin{table}[H]
  \caption{By the $P_3$-$P_3$/$P_4$ element for \eqref{u-1}  on Figure \ref{f21} grids.} \label{t3}
\begin{center}  
   \begin{tabular}{c|rr|rr}  
 \hline 
$G_i$ &  $ \|Q_h  u -   u_h \| $ & $O(h^r)$ &  $ \sqrt{\rho}\| \nabla_w( Q_h u- u_h )\|_0 $ & $O(h^r)$\\
\hline 
   &\multicolumn{4}{c}{$\rho=1$ in \eqref{b} }\\  \hline 
 4&    0.664E-05 &  4.1&    0.912E-03 &  3.0\\
 5&    0.401E-06 &  4.0&    0.114E-03 &  3.0\\
 6&    0.251E-07 &  4.0&    0.147E-04 &  3.0\\
 \hline  &\multicolumn{4}{c}{$\lambda=10^{-6}$ in \eqref{b} }\\
 \hline 
 5&    0.835E-06 &  4.0&    0.326E-06 &  3.0\\
 6&    0.523E-07 &  4.0&    0.409E-07 &  3.0\\
 7&    0.327E-08 &  4.0&    0.511E-08 &  3.0\\
  \hline  
\end{tabular} \end{center}  \end{table}

  \begin{table}[H]
  \caption{By the $P_4$-$P_4$/$P_5$ element for \eqref{u-1}  on Figure \ref{f21} grids.} \label{t4}
\begin{center}  
   \begin{tabular}{c|rr|rr}  
 \hline 
$G_i$ &  $ \|Q_h  u -   u_h \| $ & $O(h^r)$ &  $ \sqrt{\rho}\| \nabla_w( Q_h u- u_h )\|_0 $ & $O(h^r)$\\
\hline 
   &\multicolumn{4}{c}{$\rho=1$ in \eqref{b} }\\  \hline 
 4&    0.260E-06 &  5.0&    0.142E-03 &  4.0\\
 5&    0.816E-08 &  5.0&    0.893E-05 &  4.0\\
 6&    0.255E-09 &  5.0&    0.558E-06 &  4.0\\
 \hline  &\multicolumn{4}{c}{$\lambda=10^{-6}$ in \eqref{b} }\\
 \hline 
 4&    0.468E-06 &  5.0&    0.155E-06 &  4.0\\
 5&    0.147E-07 &  5.0&    0.973E-08 &  4.0\\
 6&    0.461E-09 &  5.0&    0.609E-09 &  4.0\\
  \hline  
\end{tabular} \end{center}  \end{table}

\begin{figure}[H]
 \begin{center}\setlength\unitlength{1.0pt}
\begin{picture}(360,120)(0,0)
  \put(15,108){$G_1$:} \put(125,108){$G_2$:} \put(235,108){$G_3$:} 
  \put(0,-420){\includegraphics[width=380pt]{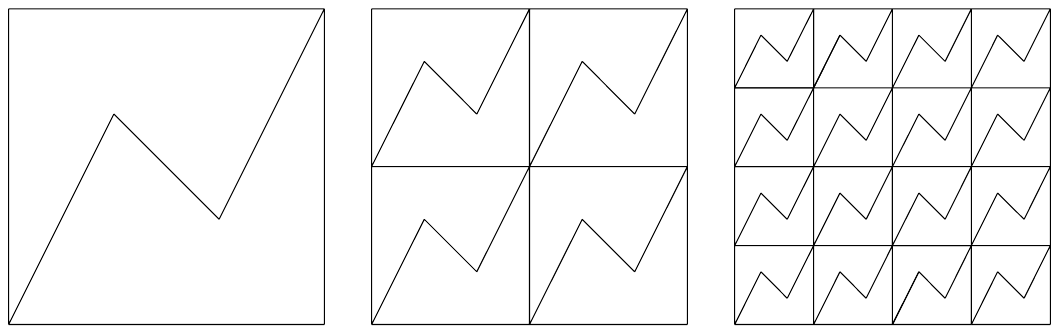}}  
 \end{picture}\end{center}
\caption{The nonconvex polyhedral grids for the computation in Tables \ref{t5}--\ref{t8},
  and in Tables \ref{t13}--\ref{t16}. }\label{f22}
\end{figure}

The solution in \eqref{u-1} is computed again by the weak Galerkin finite element
   $P_k$-$P_k$/$P_{k+2}$ (for $\{u_0, u_b\}$/$\nabla_w$), $k= 1,2,3,4$, on nonconvex polyhedral  
    grids shown in Figure \ref{f22}.
The errors and the computed orders of convergence are listed in Tables \ref{t5}--\ref{t8}.
The optimal order of convergence is also achieved in every case, but the
  results are slightly worse than those on triangular grids.
   
  \begin{table}[H]
  \caption{By the $P_1$-$P_1$/$P_3$ element for \eqref{u-1}  on Figure \ref{f22} grids.} \label{t5}
\begin{center}  
   \begin{tabular}{c|rr|rr}  
 \hline 
$G_i$ &  $ \|Q_h  u -   u_h \| $ & $O(h^r)$ &  $ \sqrt{\rho}\| \nabla_w( Q_h u- u_h )\|_0 $ & $O(h^r)$\\
\hline 
   &\multicolumn{4}{c}{$\rho=1$ in \eqref{b} }\\  \hline 
 5&    0.449E-02 &  2.0&    0.286E+00 &  1.0\\
 6&    0.113E-02 &  2.0&    0.143E+00 &  1.0\\
 7&    0.283E-03 &  2.0&    0.712E-01 &  1.0\\
 \hline  &\multicolumn{4}{c}{$\lambda=10^{-6}$ in \eqref{b} }\\
 \hline 
 5&    0.178E-02 &  2.0&    0.395E-03 &  1.0\\
 6&    0.448E-03 &  2.0&    0.199E-03 &  1.0\\
 7&    0.112E-03 &  2.0&    0.100E-03 &  1.0\\
  \hline  
\end{tabular} \end{center}  \end{table}

  \begin{table}[H]
  \caption{By the $P_2$-$P_2$/$P_4$ element for \eqref{u-1}  on Figure \ref{f22} grids.} \label{t6}
\begin{center}  
   \begin{tabular}{c|rr|rr}  
 \hline 
$G_i$ &  $ \|Q_h  u -   u_h \| $ & $O(h^r)$ &  $ \sqrt{\rho}\| \nabla_w( Q_h u- u_h )\|_0 $ & $O(h^r)$\\
\hline 
   &\multicolumn{4}{c}{$\rho=1$ in \eqref{b} }\\  \hline 
 5&    0.226E-04 &  3.0&    0.742E-02 &  1.9\\
 6&    0.284E-05 &  3.0&    0.188E-02 &  2.0\\
 7&    0.356E-06 &  3.0&    0.470E-03 &  2.0\\
 \hline  &\multicolumn{4}{c}{$\lambda=10^{-6}$ in \eqref{b} }\\
 \hline 
 5&    0.433E-04 &  3.0&    0.133E-04 &  2.0\\
 6&    0.550E-05 &  3.0&    0.334E-05 &  2.0\\
 7&    0.691E-06 &  3.0&    0.838E-06 &  2.0\\
  \hline  
\end{tabular} \end{center}  \end{table}

  \begin{table}[H]
  \caption{By the $P_3$-$P_3$/$P_5$ element for \eqref{u-1}  on Figure \ref{f22} grids.} \label{t7}
\begin{center}  
   \begin{tabular}{c|rr|rr}  
 \hline 
$G_i$ &  $ \|Q_h  u -   u_h \| $ & $O(h^r)$ &  $ \sqrt{\rho}\| \nabla_w( Q_h u- u_h )\|_0 $ & $O(h^r)$\\
\hline 
   &\multicolumn{4}{c}{$\rho=1$ in \eqref{b} }\\  \hline 
 4&    0.708E-05 &  4.1&    0.149E-02 &  3.6\\
 5&    0.454E-06 &  4.0&    0.186E-03 &  3.0\\
 6&    0.288E-07 &  4.0&    0.234E-04 &  3.0\\
 \hline  &\multicolumn{4}{c}{$\lambda=10^{-6}$ in \eqref{b} }\\
 \hline 
 4&    0.156E-04 &  3.9&    0.311E-05 &  3.2\\
 5&    0.103E-05 &  3.9&    0.394E-06 &  3.0\\
 6&    0.671E-07 &  3.9&    0.499E-07 &  3.0\\
  \hline  
\end{tabular} \end{center}  \end{table}

  \begin{table}[H]
  \caption{By the $P_4$-$P_4$/$P_6$ element for \eqref{u-1}  on Figure \ref{f22} grids.} \label{t8}
\begin{center}  
   \begin{tabular}{c|rr|rr}  
 \hline 
$G_i$ &  $ \|Q_h  u -   u_h \| $ & $O(h^r)$ &  $ \sqrt{\rho}\| \nabla_w( Q_h u- u_h )\|_0 $ & $O(h^r)$\\
\hline 
   &\multicolumn{4}{c}{$\rho=1$ in \eqref{b} }\\  \hline 
 3&    0.140E-04 &  6.7&    0.319E-02 &  5.9\\
 4&    0.313E-06 &  5.5&    0.790E-04 &  5.3\\
 5&    0.175E-07 &  ---&    0.417E-05 &  4.2\\
 \hline  &\multicolumn{4}{c}{$\lambda=10^{-6}$ in \eqref{b} }\\
 \hline 
 3&    0.186E-04 &  6.5&    0.386E-05 &  5.7\\
 4&    0.501E-06 &  5.2&    0.149E-06 &  4.7\\
 5&    0.186E-07 &  4.8&    0.884E-08 &  4.1\\
  \hline  
\end{tabular} \end{center}  \end{table}

We solve the equation \eqref{model} with parameters defined in \eqref{b}, again, where 
   the exact solution is 
\an{\label{u-2} u=\sin (\frac \pi 2 x) \sin (\frac \pi 2 y)
     (1-\t{e}^{\frac{x-1}\rho} ) (1-\t{e}^{\frac{y-1}\rho }). } 
We note that the boundary layer appears at the out-going domain boundary, $\{x=1\}$ and $\{y=1\}$.
We plot the finite solutions for \eqref{u-2} in Figure \ref{f-f-s}, where
  a thin region of function is drawn between $u_0$ and $u_b$.

\begin{figure}[H]
 \begin{center}\setlength\unitlength{1.0pt}
\begin{picture}(360,310)(0,0) 
  \put(-10,-320){\includegraphics[width=380pt]{ 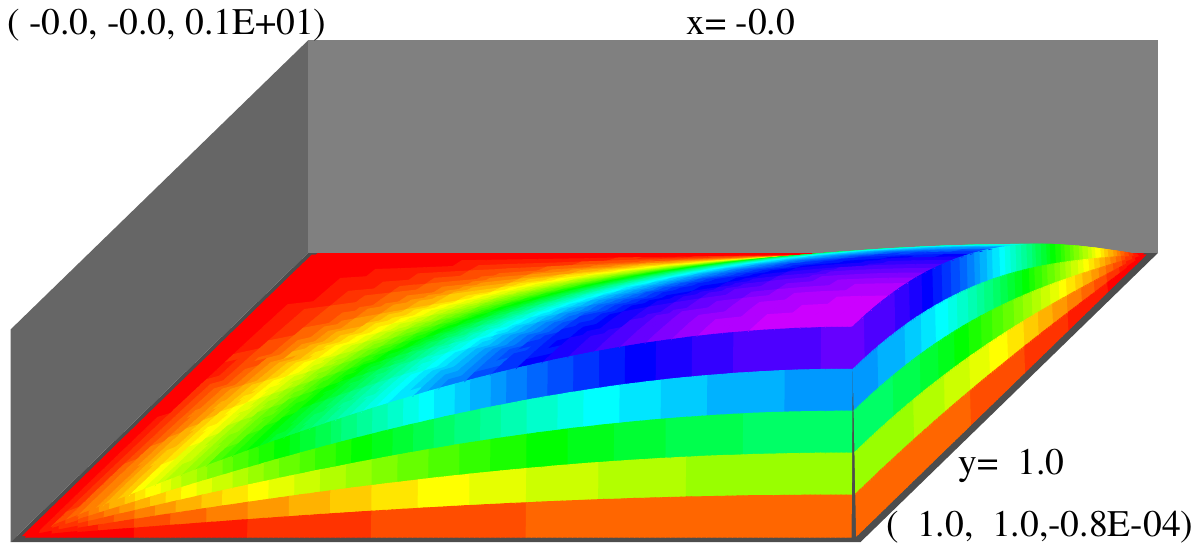}}  
  \put(-10,-145){\includegraphics[width=380pt]{ 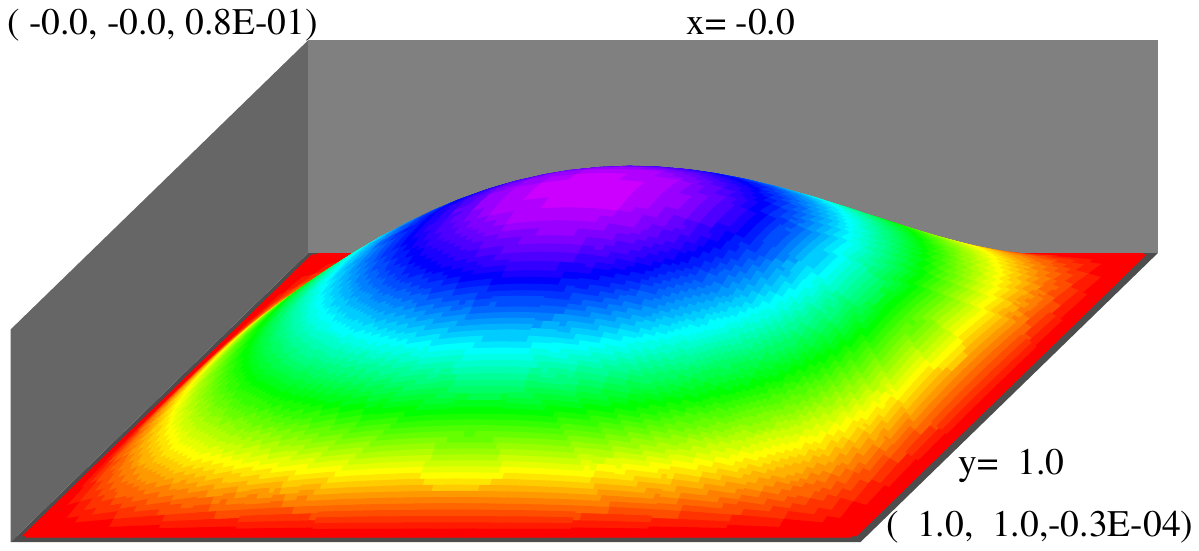}}  
 \end{picture}\end{center}
\caption{The solution of the $P_2$ WG method on grid 5 (Figure \ref{f23})
    for $\rho=1$ (top) and $\rho=10^{-9}$. }\label{f-f-s}
\end{figure}

We compute \eqref{u-2} by the weak Galerkin finite element
   $P_k$-$P_k$/$P_{k+1}$ (for $\{u_0, u_b\}$/$\nabla_w$), $k= 1,2,3,4$, on rectangular  
    grids shown in Figure \ref{f23}.
The errors and the computed orders of convergence are listed in Tables \ref{t5}--\ref{t8}.
The optimal order of convergence is also achieved in every case.
   The WG finite element method is stable for small $\rho$.

\begin{figure}[H]
 \begin{center}\setlength\unitlength{1.0pt}
\begin{picture}(360,120)(0,0)
  \put(15,108){$G_1$:} \put(125,108){$G_2$:} \put(235,108){$G_3$:} 
  \put(0,-420){\includegraphics[width=380pt]{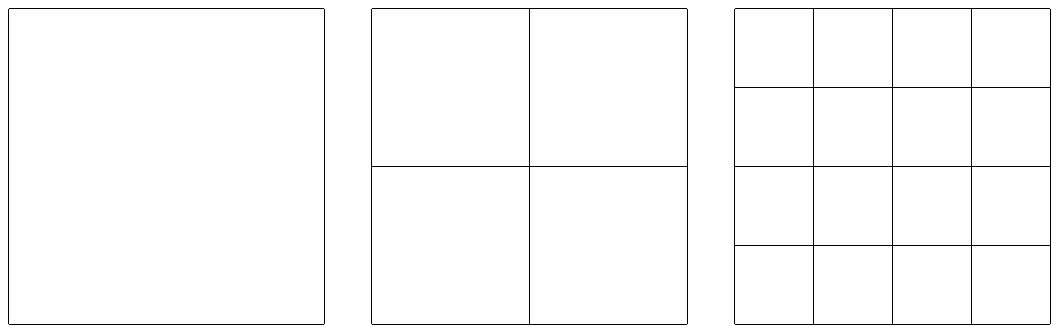}}  
 \end{picture}\end{center}
\caption{The square grids for the computation in Tables \ref{t9}--\ref{t12}. }\label{f23}
\end{figure}

  \begin{table}[H]
  \caption{By the $P_1$-$P_1$/$P_2$ element for \eqref{u-2}  on Figure \ref{f23} grids.} \label{t9}
\begin{center}  
   \begin{tabular}{c|rr|rr}  
 \hline 
$G_i$ &  $ \|Q_h  u -   u_h \| $ & $O(h^r)$ &  $ \sqrt{\rho}\| \nabla_w( Q_h u- u_h )\|_0 $ & $O(h^r)$\\
\hline 
   &\multicolumn{4}{c}{$\rho=1$ in \eqref{u-2} }\\  \hline 
 5&    0.731E-03 &  1.9&    0.120E-01 &  1.2\\
 6&    0.187E-03 &  2.0&    0.576E-02 &  1.1\\
 7&    0.471E-04 &  2.0&    0.285E-02 &  1.0\\
 \hline  &\multicolumn{4}{c}{$\lambda=10^{-9}$ in \eqref{u-2} }\\
 \hline 
 5&    0.345E-03 &  2.1&    0.219E-05 &  1.0\\
 6&    0.840E-04 &  2.0&    0.109E-05 &  1.0\\
 7&    0.207E-04 &  2.0&    0.543E-06 &  1.0\\
  \hline  
\end{tabular} \end{center}  \end{table}

  \begin{table}[H]
  \caption{By the $P_2$-$P_2$/$P_3$ element for \eqref{u-2}  on Figure \ref{f23} grids.} \label{t10}
\begin{center}  
   \begin{tabular}{c|rr|rr}  
 \hline 
$G_i$ &  $ \|Q_h  u -   u_h \| $ & $O(h^r)$ &  $ \sqrt{\rho}\| \nabla_w( Q_h u- u_h )\|_0 $ & $O(h^r)$\\
\hline 
   &\multicolumn{4}{c}{$\rho=1$ in \eqref{u-2} }\\  \hline 
 5&    0.161E-06 &  3.8&    0.649E-04 &  2.1\\
 6&    0.148E-07 &  3.4&    0.161E-04 &  2.0\\
 7&    0.163E-08 &  3.2&    0.400E-05 &  2.0\\
 \hline  &\multicolumn{4}{c}{$\lambda=10^{-9}$ in \eqref{u-2} }\\
 \hline 
 5&    0.601E-05 &  3.0&    0.420E-07 &  2.0\\
 6&    0.754E-06 &  3.0&    0.106E-07 &  2.0\\
 7&    0.106E-06 &  2.8&    0.302E-08 &  1.8\\
  \hline  
\end{tabular} \end{center}  \end{table}

  \begin{table}[H]
  \caption{By the $P_3$-$P_3$/$P_4$ element for \eqref{u-2}  on Figure \ref{f23} grids.} \label{t11}
\begin{center}  
   \begin{tabular}{c|rr|rr}  
 \hline 
$G_i$ &  $ \|Q_h  u -   u_h \| $ & $O(h^r)$ &  $ \sqrt{\rho}\| \nabla_w( Q_h u- u_h )\|_0 $ & $O(h^r)$\\
\hline 
   &\multicolumn{4}{c}{$\rho=1$ in \eqref{u-2} }\\  \hline 
 3&    0.363E-05 &  5.9&    0.252E-03 &  4.8\\
 4&    0.929E-07 &  5.3&    0.169E-04 &  3.9\\
 5&    0.447E-08 &  4.4&    0.195E-05 &  3.1\\
 \hline  &\multicolumn{4}{c}{$\lambda=10^{-9}$ in \eqref{u-2} }\\
 \hline 
 3&    0.193E-04 &  4.0&    0.343E-07 &  2.9\\
 4&    0.120E-05 &  4.0&    0.449E-08 &  2.9\\
 5&    0.784E-07 &  3.9&    0.622E-09 &  2.9\\
  \hline  
\end{tabular} \end{center}  \end{table}

  \begin{table}[H]
  \caption{By the $P_4$-$P_4$/$P_5$ element for \eqref{u-2}  on Figure \ref{f23} grids.} \label{t12}
\begin{center}  
   \begin{tabular}{c|rr|rr}  
 \hline 
$G_i$ &  $ \|Q_h  u -   u_h \| $ & $O(h^r)$ &  $ \sqrt{\rho}\| \nabla_w( Q_h u- u_h )\|_0 $ & $O(h^r)$\\
\hline 
   &\multicolumn{4}{c}{$\rho=1$ in \eqref{u-2} }\\  \hline 
 2&    0.354E-04 &  6.9&    0.150E-02 &  5.9\\
 3&    0.287E-06 &  6.9&    0.241E-04 &  6.0\\
 4&    0.364E-08 &  6.3&    0.514E-06 &  5.6\\
 \hline  &\multicolumn{4}{c}{$\lambda=10^{-9}$ in \eqref{u-2} }\\
 \hline 
 2&    0.263E-04 &  5.7&    0.305E-07 &  4.9\\
 3&    0.890E-06 &  4.9&    0.186E-08 &  4.0\\
 4&    0.383E-07 &  4.5&    0.194E-09 &  3.3\\
  \hline  
\end{tabular} \end{center}  \end{table}

Finally, \ we compute the boundary-layer solution \eqref{u-2}  by the weak Galerkin finite element
   $P_k$-$P_k$/$P_{k+2}$ (for $\{u_0, u_b\}$/$\nabla_w$), $k= 1,2,3,4$, on the nonconvex polygonal  
    grids shown in Figure \ref{f22}.
The errors and the computed orders of convergence are listed in Tables \ref{t13}--\ref{t16}.
The optimal order of convergence is also obtained in every case.

  \begin{table}[H]
  \caption{By the $P_1$-$P_1$/$P_3$ element for \eqref{u-2}  on Figure \ref{f22} grids.} \label{t13}
\begin{center}  
   \begin{tabular}{c|rr|rr}  
 \hline 
$G_i$ &  $ \|Q_h  u -   u_h \| $ & $O(h^r)$ &  $ \sqrt{\rho}\| \nabla_w( Q_h u- u_h )\|_0 $ & $O(h^r)$\\
\hline 
   &\multicolumn{4}{c}{$\rho=1$ in \eqref{u-2} }\\  \hline 
 5&    0.380E-03 &  2.0&    0.217E-01 &  1.0\\
 6&    0.958E-04 &  2.0&    0.109E-01 &  1.0\\
 7&    0.240E-04 &  2.0&    0.544E-02 &  1.0\\
 \hline  &\multicolumn{4}{c}{$\lambda=10^{-9}$ in \eqref{u-2} }\\
 \hline 
 5&    0.350E-03 &  2.0&    0.265E-05 &  1.0\\
 6&    0.881E-04 &  2.0&    0.133E-05 &  1.0\\
 7&    0.221E-04 &  2.0&    0.667E-06 &  1.0\\
  \hline  
\end{tabular} \end{center}  \end{table}

  \begin{table}[H]
  \caption{By the $P_2$-$P_2$/$P_4$ element for \eqref{u-2}  on Figure \ref{f22} grids.} \label{t14}
\begin{center}  
   \begin{tabular}{c|rr|rr}  
 \hline 
$G_i$ &  $ \|Q_h  u -   u_h \| $ & $O(h^r)$ &  $ \sqrt{\rho}\| \nabla_w( Q_h u- u_h )\|_0 $ & $O(h^r)$\\
\hline 
   &\multicolumn{4}{c}{$\rho=1$ in \eqref{u-2} }\\  \hline 
 4&    0.131E-04 &  3.1&    0.221E-02 &  1.9\\
 5&    0.163E-05 &  3.0&    0.564E-03 &  2.0\\
 6&    0.203E-06 &  3.0&    0.142E-03 &  2.0\\
 \hline  &\multicolumn{4}{c}{$\lambda=10^{-9}$ in \eqref{u-2} }\\
 \hline 
 4&    0.489E-04 &  2.9&    0.241E-06 &  2.0\\
 5&    0.623E-05 &  3.0&    0.607E-07 &  2.0\\
 6&    0.792E-06 &  3.0&    0.154E-07 &  2.0\\
  \hline  
\end{tabular} \end{center}  \end{table}

  \begin{table}[H]
  \caption{By the $P_3$-$P_3$/$P_5$ element for \eqref{u-2}  on Figure \ref{f22} grids.} \label{t15}
\begin{center}  
   \begin{tabular}{c|rr|rr}  
 \hline 
$G_i$ &  $ \|Q_h  u -   u_h \| $ & $O(h^r)$ &  $ \sqrt{\rho}\| \nabla_w( Q_h u- u_h )\|_0 $ & $O(h^r)$\\
\hline 
   &\multicolumn{4}{c}{$\rho=1$ in \eqref{u-2} }\\  \hline 
 3&    0.663E-05 &  4.8&    0.809E-03 &  4.4\\
 4&    0.414E-06 &  4.0&    0.850E-04 &  3.2\\
 5&    0.260E-07 &  4.0&    0.107E-04 &  3.0\\
 \hline  &\multicolumn{4}{c}{$\lambda=10^{-9}$ in \eqref{u-2} }\\
 \hline 
 3&    0.129E-04 &  3.8&    0.389E-07 &  3.1\\
 4&    0.927E-06 &  3.8&    0.497E-08 &  3.0\\
 5&    0.740E-07 &  3.6&    0.874E-09 &  2.5\\
  \hline  
\end{tabular} \end{center}  \end{table}

  \begin{table}[H]
  \caption{By the $P_4$-$P_4$/$P_6$ element for \eqref{u-2}  on Figure \ref{f22} grids.} \label{t16}
\begin{center}  
   \begin{tabular}{c|rr|rr}  
 \hline 
$G_i$ &  $ \|Q_h  u -   u_h \| $ & $O(h^r)$ &  $ \sqrt{\rho}\| \nabla_w( Q_h u- u_h )\|_0 $ & $O(h^r)$\\
\hline 
   &\multicolumn{4}{c}{$\rho=1$ in \eqref{u-2} }\\  \hline 
 2&    0.296E-04 &  6.7&    0.366E-02 &  5.7\\
 3&    0.443E-06 &  6.1&    0.714E-04 &  5.7\\
 4&    0.135E-07 &  5.0&    0.289E-05 &  4.6\\
 \hline  &\multicolumn{4}{c}{$\lambda=10^{-9}$ in \eqref{u-2} }\\
 \hline 
 2&    0.213E-04 &  6.5&    0.722E-07 &  5.7\\
 3&    0.537E-06 &  5.3&    0.271E-08 &  4.7\\
 4&    0.390E-07 &  3.8&    0.389E-09 &  ---\\
  \hline  
\end{tabular} \end{center}  \end{table}

\end{document}